\documentclass[reqno]{amsart}
\usepackage{mathrsfs}
\usepackage{amssymb}
\usepackage{amscd}
\theoremstyle{plain}
\newtheorem{theorem}{Theorem}[section]
\newtheorem{lemma}{Lemma}[section]

\newtheorem{definition}{Definition}[section]
\newtheorem{conjecture}{Conjecture}[section]

\begin{document}
\title{A note on the image of polynomials and Waring type problems on upper triangular matrix algebras}
\thanks{}
\maketitle
\begin{center}
Qian Chen\\
Department of Mathematics, Shanghai Normal University,
Shanghai 200234, China.\\
Email address: qianchen0505@163.com
\end{center}
\maketitle

\begin{abstract}
In the present paper we shall obtain a result on the image of polynomials with zero constant term on upper triangular matrix algebras over an algebraically closed field. This is a supplement to a result obtained by Panja and Prasad recently.
\end{abstract}

\section{Introduction}
Let $K$ be a field. Let $m\geq 1$ be an
integer and let $n\geq 2$ be an integer. By $K\langle x_1,\ldots,x_m\rangle$ we denote the free
$K$-algebra generated by non-commuting variables $x_1,\ldots,x_m$
and refer to the elements of $K\langle x_1,\ldots,x_m\rangle$ as
polynomials.

The old and famous Lvov-Kaplansky conjecture asserts:

\begin{conjecture}\cite{Do}\label{Con1}
Let $p$ be a multilinear polynomial. Then the set of values
of $p$ on the matrix algebra $M_n(K)$ over a field $K$ is a vector space.
\end{conjecture}

The images of multilinear polynomials of degree up to four on matrix algebras have been discussed (see \cite{Albert,BW,Mes,Sho}).

For an arbitrary polynomial with zero constant term, the question was discussed for the case when $K$ is a
finite field by Chuang \cite{Chuang}. Later Chuang's result was generalized by Kulyamin \cite{Kul1,Kul2} for graded algebras.

In 2012 Kanel-Belov, Malev, and Rowen \cite{Rowen1} gave a complete
description of the image of semi-homogeneous polynomials on the
algebra of $2\times 2$ matrices  over a quadratically closed field. As a consequence, they solved Conjecture \ref{Con1} for $n=2$ (see
\cite[Theorem 2]{Rowen1}). In 2016 they gave a complete description of the image of a multilinear polynomial which is trace vanishing
on $3\times 3$ matrices over a field $K$ (see \cite[Theorem 4]{Rowen2}). Some results related to Conjecture \ref{Con1}
have been obtained (see \cite{Rowen4,Dyk,DCM, Malev1, Malev2,Vitas1}). In 2020 Bre\v{s}ar \cite{Bre} presented some rough approximate versions of Conjecture \ref{Con1}, or
at least as an attempt to approach this conjecture from a different perspective.

There has been a growing interest about the Waring problem for matrix
algebras. For example, Bre\v{s}ar and \v{S}emrl \cite{B1} proved that any traceless matrix can be written as sum of two
matrices from $f(M_n(C))-f(M_n(C))$, where $C$ is the complex field and $f$ is neither an identity nor a central polynomial for
$M_n(C)$. Recently, they have also proved that if $\alpha_1, \alpha_2,\alpha_3\in C\setminus\{0\}$ and
$\alpha_1+\alpha_2+\alpha_3=0$, then any traceless matrix over $C$ can be written as $\alpha_1A_1+\alpha_2A_2+\alpha_3A_3$, where $A_i\in f(M_n(C))$ (see \cite{B2}).

In attempts to approach Conjecture \ref{Con1}, some
variations of it have been studied extensively. For example, the images of
multilinear polynomials of small degree on Lie Algebras \cite{An,SP}
and Jordan Algebras \cite{Ma} have been discussed. In 2021, Malev \cite{Malev2} gave a complete description of the images of multilinear polynomials evaluated on the quaternion algebra. In 2021, Vitas \cite{Vitas2} proved for any nonzero multilinear polynomial $p$,
that if $\mathcal{A}$ is an algebra with a surjective inner derivation, such as the Weyl algebra, then $p(\mathcal{A})=\mathcal{A}$.
Recently, Kanel-Belov, Malev, Pines, and Rowen \cite{Rowen2022} investigated the images of multilinear and semihomogeneous polynomials on the
algebra of octonions. For the most recent results on images of polynomials on finite dimensional algebras we recommend the survey paper \cite{survey}.

By $T_n(K)$ we denote the set of all $n\times n$ upper triangular
matrices over $K$. By $T_n(K)^{(0)}$ we denote the set of all
$n\times n$ strictly upper triangular matrices over $K$. More generally, if $t\geq 0$, the set of all upper triangular
matrices whose entries $(i,j)$ are zero, for $j- i\leq t$, will be
denoted by $T_n(K)^{(t)}$. It is easy to check that $J=T_n(K)^{(0)}$ and $J^t=T_n(K)^{(t-1)}$, where $t\geq 1$ and $J$ is the Jacobson radical of $T_n(K)$ (see \cite[Example 5.58]{Bre1}).

In 2019 Fagundes \cite{Fag} gave a complete description of the images
of multilinear polynomials on $T_n(K)^{(0)}$. In the same year, Fagundes and Mello \cite{FM} discussed the images
of multilinear polynomials of degree up to four on $T_n(K)$. They proposed the following important variation of Conjecture \ref{Con1}:

\begin{conjecture}\cite[Conjecture 1]{FM}\label{Con}
The image of a multilinear polynomial over a field $K$ on the upper
triangular matrix algebra $T_n(K)$ is always a vector space.
\end{conjecture}

In 2019 Wang \cite{Wang2019} gave a positive answer of
Conjecture \ref{Con} for $n=2$ (see also \cite{Wang20191}). In 2021 Mello \cite{Mello1} solved Conjecture \ref{Con} for $n=3$ and $|K|\geq 3$. In 2022 Gargate and Mello \cite{Mello} solved Conjecture \ref{Con} on an infinite field. In the same year, Luo and Wang
\cite{Wang2022} solved Conjecture \ref{Con} under some mild conditions on the ground field $K$. More precisely, they obtained the following result:

\begin{theorem}\cite[Theorem 1.1]{Wang2022}\label{TTT1}
Let $m\geq 1$ be an integer. Let $n\geq 2$ be an integer. Let $K$ be a field. Let $p(x_1,\ldots,x_m)$ be a nonzero
multilinear polynomial in non-commutative variables over $K$. Suppose that $|K|\geq\frac{n(n-1)}{2}$. We have that $p(T_n(K))=T_n(K)^{(t)}$ for some integer $-1\leq t\leq \frac{m}{2}$.
\end{theorem}

In 2022 Fagundes and Koshlukov \cite{FK} investigated the image of multilinear graded polynomials on upper triangular matrix algebras. Recently Luo and Chen \cite{LuoChen} gave a complete description of the image of linear polynomials on upper triangular matrix algebras, which improves Theorem \ref{TTT1} and some results in \cite{FK}.

In 2021 Zhou and Wang \cite{ZhouWang} gave a complete description of
the image of complete homogeneous polynomials on $2\times 2$ upper triangular matrix algebras over an algebraically closed field $K$. In the same year, Wang, Zhou, and Luo \cite{WangZL} gave a complete description of
the image of polynomials with zero constant term on $2\times 2$ upper triangular matrix algebras over an algebraically closed field $K$. In 2022 Chen, Luo, Wang \cite{ChenLuoWang} gave a complete description of the image of polynomials with zero constant term on $3\times 3$ upper triangular matrix algebras over an algebraically closed field $K$. In 2023 Panja and Prasad \cite{PP} discussed the image of polynomials with zero constant term and Waring type problems on upper triangular matrix algebras over an algebraically closed field. More precisely, they obtained the following  main result:

\begin{theorem}\cite[Theorem 5.18]{PP}\label{PP}
Let $n\geq 2$ and $m\geq 1$ be integers. Let $p(x_1,\ldots,x_m)$ be a
polynomial with zero constant term in non-commutative variables over an algebraically closed field $K$. Set $r=$\emph{ord}$(p)$. Then one of the following statements holds.
\begin{enumerate}
\item[(i)] Suppose that $r=0$. We have that $p(T_n(K))$ is a dense subset of $T_n(K)$ (with respect to the Zariski topology);
\item[(ii)] Suppose that $r=1$. We have that $p(T_n(K))=T_n(K)^{(0)}$;
\item[(iii)] Suppose that $1<r<n-1$. We have that $p(T_n(K))\subseteq T_n(K)^{(r-1)}$, and equality might not hold in general.
Furthermore, for every $n$ and $r$ there exists $d$ such that each element of $T_n(K)^{(r-1)}$ can be written as a sum of $d$ many elements from $p(T_n(K))$;
\item[(iv)] Suppose that $r=n-1$. We have that $p(T_n(K))=T_{n}(K)^{(n-2)}$;
\item[(v)] Suppose that $r\geq n$. We have that $p(T_n(K))=\{0\}$.
\end{enumerate}
\end{theorem}

We remark that the statement (iii) in Theorem \ref{PP} gave a connection between $T_n(K)^{(r-1)}$ and $p(T_n(K))$ in the case of $1<r<n-1$, which is the Waring type problems.

In the present paper, we shall give a complete description of $p(T_n(K))$ in the case of $1<r<n-1$. More precisely, we shall prove the following main result.

\begin{theorem}\label{T1}
Let $n\geq 2$ and $m\geq 1$ be integers. Let $p(x_1,\ldots,x_m)$ be a
polynomial with zero constant term in non-commutative variables over an algebraically closed field $K$. Set $r=$\emph{ord}$(p)$. Suppose that $1<r<n-1$. We have that $p(T_n(K))$ is a dense subset of $T_n(K)^{(r-1)}$ (with respect to the Zariski topology).
\end{theorem}

We organize the paper as follows: In Section $2$ we shall give some prelimiaries. We shall define a concept of compatible subsets, which is crucial for the proof of our main result. In Section $3$ we shall give the proof of our main result.

\section{prelimiaries}

By $\mathcal{N}$ we denote the set of all positive integers. Let $K$ be a field. For $n\in \mathcal{N}$ with $n\geq 2$, we can write
\[
T_n(K)=\left(
\begin{array}{cc}
T_{n-1}(K) & K^{n-1}\\
 & K
\end{array} \right).
\]

Let $\mathcal{A}$ be an algebra over $K$. Denote by $\mathcal{T}(\mathcal{A})$ the set of all polynomial identities of $\mathcal{A}$. Then it is easy to see that
\[
\mathcal{T}(K)\supset \mathcal{T}(T_2(K))\supset \mathcal{T}(T_3(K))\supset\cdots.
\]
So, given a polynomial $p$ in non-commutative variables over $K$, we define its \textbf{order}
as the least integer $m$ such that $p\in \mathcal{T}(T_m(K))$ but $p\not\in \mathcal{T}(T_{m+1}(K))$.
Note that $T_1(K)=K$. A polynomial $p$ has order $0$ if $p\not\in \mathcal{T}(K)$. We denote the order of $p$ by $\mbox{ord}(p)$. For a detailed introduction of the order of polynomials
we refer the reader to the book \cite[Chapter 5]{Drensky}.

We remark that the structure of $\mathcal{T}(T_n(K))$ is essentially known. For instance,
the basis of $\mathcal{T}(T_n(K))$ has been described in \cite[Theorem 5.4]{Drensky},
if $K$ is infinite field. A similar insight holds in the context of finite fields.

Let $K$ be an algebraically closed field. Set
\[
K^n=\{(a_1,\ldots,a_n)~|~ a_{i}\in K\}.
\]
We call $K^n$ the \emph{affine $n$-space} over $K$. We assume that $T_n(K)=K^{\frac{n(n+1)}{2}}$, an affine $\frac{n(n+1)}{2}$-space over $K$.

Let $\mathcal{A}=K[y_1,\ldots,y_n]$ be the polynomial algebra in $n$ variables over $K$. For $T\subseteq \mathcal{A}$, we define
\[
Z(T)=\{P\in K^n~|~f(P)=0\quad\mbox{for all $f\in T$}\}.
\]
We call $Z(T)$ the \emph{zero set} of $T$. A subset $Y$ of $K^n$ is an \emph{algebraic set} if there exists a subset $T$ of
$\mathcal{A}$ such that $Y=Z(T)$ (see \cite[Definition 1.1.1]{Robin}).

We define the \emph{Zariski topology} on $K^n$ by taking the open subsets to be the complements of the algebraic sets (see \cite[Definition 1.1.2]{Robin}).

It is easy to check that a subset $Y$ of $K^n$ is a dense subset of $K^n$ (with respect to the Zariski topology) if for any $P\in K^n$ and any open subset $U$ of $K^n$ with $P\in U$
we have $Y\cap U\neq \emptyset$.

Let $p(x_1,\ldots,x_m)$ be a polynomial with zero constant term over
$K$. We can write
\begin{equation}\label{e1}
p(x_1,\ldots,x_m)=\sum\limits_{k=1}^d\left(\sum\limits_{1\leq i_1,i_2,\ldots,i_k\leq m}\lambda_{i_1i_2\cdots i_k}x_{i_1}x_{i_2}\cdots x_{i_k}\right),
\end{equation}
where $\lambda_{i_1i_2\cdots i_k}\in K$ and $d$ is the degree of $p$. For any $u_i=(a_{jk}^{(i)})\in T_n(K)$, $i=1,\ldots,m$, we set
\[
\bar{a}_{jj}=(a_{jj}^{(1)},\ldots,a_{jj}^{(m)}),
\]
where $j=1,\ldots,n$. One easily check from (\ref{e1}) that

\begin{equation}\label{e2}
p(u_1,\ldots,u_m)=\left(
\begin{array}{cccc}
p(\bar{a}_{11}) & p_{12} & \ldots & p_{1n}\\
0 & p(\bar{a}_{22}) & \ldots & p_{2n}\\
 \vdots & \vdots & \ddots & \vdots\\
 0 & 0 & \ldots & p(\bar{a}_{nn})
\end{array} \right),
\end{equation}
where
\[
p_{st}=\sum\limits_{k=1}^{t-s}\left(\sum\limits_{\substack{s=j_1<j_2<\cdots <j_{k+1}=t\\1\leq i_1,\ldots,i_k\leq m}}
p_{i_1i_2\cdots i_k}(\bar{a}_{j_1j_1},\ldots,\bar{a}_{j_{k+1}j_{k+1}})a_{j_1j_2}^{(i_1)}\cdots a_{j_{k}j_{k+1}}^{(i_k)}\right)
\]
for all $1\leq s<t\leq n$, where $p_{i_1,\ldots,i_k}(z_1,\ldots,z_{m(k+1)})$, $1\leq i_1,i_2,\ldots,i_k\leq m$, $k=1,\ldots,n-1$, is a
commutative polynomial over $K$. We shall give a proof of the display (\ref{e2}) in the next section.

For $1\leq s\leq t\leq n$, we set
\[
U_{s,t}=\{(j,k)\in\mathcal{N}\times\mathcal{N}~|~s\leq j\leq k\leq t\}.
\]

It follows from (\ref{e2}) that

\[
p_{st}=\sum\limits_{k=1}^{t-s}\left(\sum\limits_{\substack{s=j_1<j_2<\cdots <j_{k+1}=t\\1\leq i_1,\ldots,i_k\leq m}}
p_{i_1i_2\cdots i_k}(\bar{x}_{j_1j_1},\ldots,\bar{x}_{j_{k+1}j_{k+1}})x_{j_1j_2}^{(i_1)}\cdots x_{j_{k}j_{k+1}}^{(i_k)}\right)
\]
where $1\leq s<t\leq n$, is a polynomial with zero constant term on the variables
\[
\left\{x_{jk}^{(i)}~|~(j,k)\in U_{s,t},i=1,\ldots,m\right\}.
\]

\begin{definition}\label{D1}
Let $K$ be a field. Let
\[
V_1=\{a_{i}\in K~|~i\in I_1\}\quad\mbox{and}\quad V_2=\{b_{i}\in K~|~i\in I_2\},
\]
where $I_1,I_2$ are two index sets. We call $V_1$ and $V_2$ compatible subsets of $K$ if $a_i=b_i$ for all $i\in I_1\cap I_2$.
 More generally, let $V_1,\ldots,V_m$ be the subsets of $K$. We call $V_i$, $i=1,\ldots,m$, compatible subsets of $K$
 if $V_i$ and $V_j$ are compatible for any $i,j\in\{1,\ldots,m\}$.
\end{definition}

Let
\[
V_{s,t}=\left\{a_{jk}^{(i)}\in K~|~(j,k)\in U_{s,t},i=1,\ldots,m\right\},
\]
where $1\leq s\leq t\leq n$, be  subsets of $K$. It is clear that $V_{s,t}$, $1\leq s\leq t\leq n$, are compatible subsets of $K$ if and only if $u_i=(a_{jk}^{(i)})\in T_n(K)$, $i=1,\ldots,m$.

\section{The proof of the main results}

We begin with the following simple result, which will be used in the proof of the next result. We omit its proof for brevity.

\begin{lemma}\label{L1}
Let $k\geq 2$ be an integer. Let $u_i=(a_{jk}^{(i)})\in T_n(K)$, where $i=1,\ldots,m$. For any $1\leq i_1,\ldots,i_w\leq m$, we have that
\[
u_{i_1}\cdots u_{i_w}=\left(
\begin{array}{cccc}
m_{11} & m_{12} & \ldots & m_{1n}\\
0 & m_{22} & \ldots & m_{2n}\\
 \vdots & \vdots & \ddots & \vdots\\
 0 & 0 & \ldots & m_{nn}
\end{array} \right),
\]
where
\[
m_{st}=\sum\limits_{s=j_1\leq j_2\leq \cdots \leq j_{w+1}=t}a_{j_1j_2}^{(i_1)}\cdots a_{j_{w}j_{w+1}}^{(i_w)}
\]
for all $1\leq s\leq t\leq n$.
\end{lemma}

We now give the proof of the display (\ref{e2}).

\begin{lemma}\label{L2}
Let $n\geq 2$ and $m\geq 1$ be integers. Let $p(x_1,\ldots,x_m)$ be a polynomial with zero constant term in
non-commutative variables over a field $K$. For $u_i=(a_{jk}^{(i)})\in T_n(K)$, $i=1,\ldots,m$, we have that
\[
p(u_1,\ldots,u_m)=\left(
\begin{array}{cccc}
p(\bar{a}_{11}) & p_{12} & \ldots & p_{1n}\\
0 & p(\bar{a}_{22}) & \ldots & p_{2n}\\
 \vdots & \vdots & \ddots & \vdots\\
 0 & 0 & \ldots & p(\bar{a}_{nn})
\end{array} \right),
\]
where
\[
p_{st}=\sum\limits_{k=1}^{t-s}\left(\sum\limits_{\substack{s=j_1<j_2<\cdots <j_{k+1}=t\\1\leq i_1,\ldots,i_k\leq m}}
p_{i_1i_2\cdots i_k}(\bar{a}_{j_1j_1},\ldots,\bar{a}_{j_{k+1}j_{k+1}})a_{j_1j_2}^{(i_1)}\cdots a_{j_{k}j_{k+1}}^{(i_k)}\right)
\]
for all $1\leq s<t\leq n$, where $p_{i_1,\ldots,i_k}(z_1,\ldots,z_{m(k+1)})$ is a commutative polynomial over $K$.
\end{lemma}

\begin{proof}
In view of Lemma \ref{L1}, we get from (\ref{e1}) that
\begin{eqnarray}\label{e3}
\begin{split}
p(u_1,\ldots,u_m)&=\sum\limits_{k=1}^d\left(\sum\limits_{1\leq i_1,\ldots,i_k\leq m}\lambda_{i_1\cdots i_k}u_{i_1}\cdots u_{i_k}\right)\\
&=\sum\limits_{k=1}^d\left(\sum\limits_{1\leq i_1,\ldots,i_k\leq m}\lambda_{i_1\cdots i_k}\left(
\begin{array}{cccc}
m_{11} & m_{12} & \ldots & m_{1n}\\
0 & m_{22} & \ldots & m_{2n}\\
 \vdots & \vdots & \ddots & \vdots\\
 0 & 0 & \ldots & m_{nn}
\end{array} \right)\right)\\
&=\left(
\begin{array}{cccc}
p_{11} & p_{12} & \ldots & p_{1n}\\
0 & p_{22} & \ldots & p_{2n}\\
 \vdots & \vdots & \ddots & \vdots\\
 0 & 0 & \ldots & p_{nn}
\end{array} \right)
\end{split}
\end{eqnarray}
where
\begin{eqnarray*}
\begin{split}
p_{st}&=\sum\limits_{k=1}^d\left(\sum\limits_{1\leq i_1,\ldots,i_k\leq m}\lambda_{i_1\cdots i_k}m_{st}\right)\\
&=\sum\limits_{k=1}^d\left(\sum\limits_{1\leq i_1,\ldots,i_k\leq m}\lambda_{i_1\cdots i_k}\left(
\sum\limits_{s=j_1\leq j_2\leq \cdots \leq j_{k+1}=t}a_{j_1j_2}^{(i_1)}\cdots a_{j_{k}j_{k+1}}^{(i_k)}\right)\right)\\
&=\sum\limits_{k=1}^{d}\left(\sum\limits_{\substack{s=j_1\leq j_2\leq \cdots \leq j_{k+1}=t\\1\leq i_1,\ldots,i_k\leq m}}\lambda_{i_1i_2\cdots i_k}a_{j_1j_2}^{(i_1)}\cdots a_{j_{k}j_{k+1}}^{(i_k)}\right),
\end{split}
\end{eqnarray*}
where $1\leq s\leq t\leq n$. For any $1\leq s\leq n$, we get from (\ref{e3}) that
\begin{eqnarray*}
\begin{split}
p_{ss}&=\sum\limits_{k=1}^{d}\left(\sum\limits_{1\leq i_1,\ldots,i_k\leq m}\lambda_{i_1i_2\cdots i_k}a_{ss}^{(i_1)}\cdots a_{ss}^{(i_k)}\right)\\
&=p(\bar{a}_{ss}).
\end{split}
\end{eqnarray*}
For any $1\leq s<t\leq n$, we get from (\ref{e3}) that
\begin{eqnarray*}
\begin{split}
p_{st}&=\sum\limits_{k=1}^{d}\left(\sum\limits_{\substack{s=j_1\leq j_2\leq \cdots \leq j_{k+1}=t\\1\leq i_1,\ldots,i_k\leq m}}\lambda_{i_1i_2\cdots i_k}a_{j_1j_2}^{(i_1)}\cdots a_{j_{k}j_{k+1}}^{(i_k)}\right)\\
&=\sum\limits_{k=1}^{t-s}\left(\sum\limits_{\substack{s=j_1<j_2<\cdots <j_{k+1}=t\\1\leq i_1,\ldots,i_k\leq m}}
p_{i_1i_2\cdots i_k}(\bar{a}_{j_1j_1},\ldots,\bar{a}_{j_{k+1}j_{k+1}})a_{j_1j_2}^{(i_1)}\cdots a_{j_{k}j_{k+1}}^{(i_k)}\right),
\end{split}
\end{eqnarray*}
where $p_{i_1,\ldots,i_k}(z_1,\ldots,z_{m(k+1)})$ is a commutative polynomial over $K$. This proves the result.
\end{proof}

The following technical result comes from both \cite[Lemma 3.1]{ChenLuoWang} and \cite[Lemma 3.1]{WangZL}.
We omit its proof for brevity.

\begin{lemma}\label{L3}
Let $m\geq 1$ be integer. Let $p(x_1,\ldots,x_m)$ be a
 polynomial with zero constant term over an algebraically closed field $K$. Suppose that $p(K)\neq\{0\}$. Then $p(K)=K$.
\end{lemma}

\begin{lemma}\label{L4}
Let $m\geq 1$ be an integer. Let $p(x_1,\ldots,x_m)$ be a polynomial with zero constant term over an algebraically closed field $K$.
Let $p_{i_1,\ldots,i_k}(z_1,\ldots,z_{m(k+1)})$ be a commutative polynomial of $p$ in (\ref{e2}), where $1\leq i_1,\ldots,i_k\leq m$, $1\leq k\leq n-1$. Suppose that \emph{ord}$(p)=r$, $1\leq r\leq n-1$. We have that
\begin{enumerate}
\item[(i)] $p(K)=\{0\}$;
\item[(ii)] $p_{i_1,\ldots,i_k}(K)=\{0\}$ for all $1\leq i_1,\ldots,i_k\leq m$, where $k=1,\ldots,r-1$ and $r\geq 2$. that is, $p(T_n(K))\subseteq T_n(K)^{(r-1)}$;
\item[(iii)] $p_{i_1,\ldots,i_r}(K)\neq\{0\}$ for some $1\leq i_1,\ldots,i_r\leq m$.
\end{enumerate}
\end{lemma}

\begin{proof}
Since $r\geq 1$, we get that the statement (i) holds true. We now claim that the statement (ii) holds true. We assume that
$r\geq 2$. Suppose on the contrary that
\[
p_{i_1,\ldots,i_s}(K)\neq \{0\}
\]
for some $1\leq i_1,\ldots,i_s\leq m$, where $1\leq s\leq r-1$. Then there exist $\bar{b}_{j}\in K^m$, where $j=1,\ldots,s+1$ such that
\[
p_{i_1,\ldots,i_s}(\bar{b}_{1},\ldots,\bar{b}_{s+1})\neq 0.
\]
We take $u_i=(a_{jk}^{(i)})\in T_{s+1}(K)$, $i=1,\ldots,m$, where
\[
\left\{
\begin{aligned}
\bar{a}_{tt}&=\bar{b}_{t},\quad\mbox{for all $t=1,\ldots,s+1$};\\
a_{t,t+1}^{(i_t)}&=1,\quad\mbox{for all $t=1,\ldots,s$};\\
a_{jk}^{(i)}&=0,\quad\mbox{otherwise}.
\end{aligned}
\right.
\]
We get that
\[
p_{1,s+1}=p_{i_1,i_2,\ldots,i_s}(\bar{b}_{1},\ldots,\bar{b}_{s+1})\neq 0.
\]
This implies that $p(T_{s+1}(K))\neq\{0\}$, a contradiction. This proves the statement (ii).

We finally claim that the statement (iii) holds true. Note that $p(T_{1+r}(K))\neq\{0\}$. Thus,
we have that there exist $u_{i}=(a_{jk}^{(i)})\in T_{1+r}(K)$, $i=1,\ldots,m$, such that
\[
p(u_1,\ldots,u_m)=(p_{st})\neq 0.
\]
In view of the statement (ii) we get that
\[
p_{1,r+1}=\sum\limits_{\substack{1=j_1<j_2< \cdots <j_{r+1}=r+1\\1\leq i_1,\ldots,i_r\leq m}}p_{i_1i_2\cdots i_r}(\bar{a}_{j_1j_1},\ldots,\bar{a}_{j_{r+1}j_{r+1}})a_{j_1j_2}^{(i_1)}\cdots a_{j_{r}j_{r+1}}^{(i_r)}\neq 0.
\]
This implies that $p_{i_1,\ldots,i_r}(K)\neq \{0\}$ for some $1\leq i_1,\ldots,i_r\leq m$. This proves the statement (iii). The proof of the result is complete.
\end{proof}

The following technical result will be used in the proof of the main result.

\begin{lemma}\label{L5}
Let $m,n,s$ be integers with $1\leq s\leq n$. Let $p(x_1,\ldots,x_s)$ be a nonzero commutative polynomials over an algebraically closed field $K$.  We have that there exist $a_1,\ldots,a_n\in K$ such that
\begin{eqnarray*}
\begin{split}
p(a_{i_1},\ldots,a_{i_s})\neq 0
\end{split}
\end{eqnarray*}
for all $\{i_1,\ldots,i_s\}\subseteq \{1,\ldots,n\}$.
\end{lemma}

\begin{proof}
We set
\[
f(x_1,\ldots,x_n)=\prod_{\{i_1,\ldots,i_s\}\subseteq \{1,\ldots,n\}}p(x_{i_1},\ldots,x_{i_s}).
\]
It is clear that $f\neq 0$. In view of \cite[Theorem 2.19]{Jac} we have that there exist $a_1,\ldots,a_n\in K$ such that
\[
f(a_1,\ldots,a_n)\neq 0.
\]
This implies that
\[
p(a_{i_1},\ldots,a_{i_s})\neq 0
\]
for all $\{i_1,\ldots,i_s\}\subseteq \{1,\ldots,n\}$. This proves the result.
\end{proof}

The following well known result comes from \cite[Lemma 2.11]{FK}.

\begin{lemma}\label{L6}
Let $m,t\geq 1$. Let $p_i(x_1,\ldots,x_m)$ be a polynomial over an algebraically closed field $K$, where $i=1,\ldots,t$. Suppose that $p_i(K)\neq\{0\}$ for all $i=1,\ldots,t$. Then there exist $a_1,\ldots,a_m\in K$ such that
\[
p_i(a_1,\ldots,a_m)\neq 0
\]
for all $i=1,\ldots,t$.
\end{lemma}

We now give the proof of our main result.

\begin{proof}[The proof of Theorem \ref{T1}]
In view of Lemma \ref{L4}(ii) we note that $p(T_n(K))\subseteq T_n(K)^{(r-1)}$. We assume that
\[
T_n(K)^{(r-1)}=K^{d},
 \]
an affine $d$-space over $K$, where $d=\frac{(n-r)(n-r+1)}{2}$. For any $P'\in K^{d}$ and an open subset $U$ of $K^{d}$ with $P'\in U$, we write
 \[
 P'=(a_{1,r+1}',a_{2,r+2}',\ldots,a_{n-r,n}',a_{1,r+2}',\ldots,a_{n-r-1,n}',\ldots,a_{1,n}').
 \]

Since $U$ is an open subset of $K^d$, we have that there exists a subset $T$ of $K[y_1,\ldots,y_d]$ such that
\[
U=K^d\setminus Z(T).
\]
Since $P'\in U$ we have that there exists $f(y_1,\ldots,y_d)\in T$ such that $f(P')\neq 0$. That is,
\begin{equation}\label{e6}
f(a_{1,r+1}',a_{2,r+2}',\ldots,a_{n-r,n}',a_{1,r+2}',a_{2,r+3}',\ldots,a_{n-r-1,n}',\ldots,a_{1,n}')\neq 0.
\end{equation}

For any $u_i=(a_{jk}^{(i)})\in T_n(K)$, $i=1,\ldots,m$, we get from both Lemma \ref{L4}(ii) and (\ref{e2}) that
\begin{equation}\label{eeee1}
p_{st}=\sum\limits_{k=r}^{t-s}\left(\sum\limits_{\substack{s=j_1<j_2<\cdots <j_{k+1}=t\\1\leq i_1,\ldots,i_k\leq m}}
p_{i_1i_2\cdots i_k}(\bar{a}_{j_1j_1},\ldots,\bar{a}_{j_{k+1}j_{k+1}})a_{j_1j_2}^{(i_1)}\cdots a_{j_{k}j_{k+1}}^{(i_k)}\right)
\end{equation}
for all $1\leq s<t\leq n$ with $t-s\geq r$. In view of Lemma \ref{L4}(iii) we have that
\[
p_{i_1',\ldots,i_r'}(K)\neq\{0\},
\]
for some $i_1',\ldots,i_r'\in\{1,\ldots,m\}$. It follows from Lemma \ref{L5} that
there exist $\bar{b}_1,\ldots,\bar{b}_n\in K^m$ such that
\begin{equation}\label{e7}
p_{i_1',\ldots,i_r'}(\bar{b}_{j_1},\ldots,\bar{b}_{j_{r+1}})\neq 0
\end{equation}
for all $\{j_1,\ldots,j_{r+1}\}\subseteq\{1,\ldots,n\}$. Taking
\[
\bar{a}_{jj}=\bar{b}_{j}
\]
in (\ref{eeee1}) for all $j=1,\ldots,n$, we get that
\begin{equation}\label{t1}
p_{st}=\sum\limits_{k=r}^{t-s}\left(\sum\limits_{\substack{s=j_1<j_2<\cdots <j_{k+1}=t\\1\leq i_1,\ldots,i_k\leq m}}
p_{i_1i_2\cdots i_k}(\bar{b}_{j_1},\ldots,\bar{b}_{j_{k+1}})a_{j_1j_2}^{(i_1)}\cdots a_{j_{k}j_{k+1}}^{(i_k)}\right)
\end{equation}
for all $1\leq s<t\leq n$ with $t-s\geq r$.  We set
\[
\bar{U}_{s,t}=\{(j,k)\in U_{s,t}~|~\mbox{$j\neq k$ and $(t-s)-(k-j)\geq r-1$}\}.
\]

It is clear that
\[
p_{st}=\sum\limits_{k=r}^{t-s}\left(\sum\limits_{\substack{s=j_1<j_2<\cdots <j_{k+1}=t\\1\leq i_1,\ldots,i_k\leq m}}
p_{i_1i_2\cdots i_k}(\bar{b}_{j_1},\ldots,\bar{b}_{j_{k+1}})x_{j_1j_2}^{(i_1)}\cdots x_{j_{k}j_{k+1}}^{(i_k)}\right),
\]
where $1\leq s<t\leq n$ with $t-s\geq r$, is a polynomial with zero constant term on the variables
\[
\left\{x_{jk}^{(i)}~|~(j,k)\in \bar{U}_{s,t},i=1,\ldots,m\right\}.
\]

For any $(s,r+s+t)$, where $1\leq s<r+s+t\leq n$, we rewrite (\ref{e6}) as follows:
\begin{equation}\label{e8}
f(a_{1,r+1}',\ldots,a_{s_1,t_1}',a_{s,r+s+t}',a_{s_2,t_2}',\ldots,a_{1,n}')\neq 0.
\end{equation}
We claim that there exist the following compatible subsets of $K$
\begin{eqnarray*}
\begin{split}
V_{1,r+1}&=\left\{a_{jk}^{(i)}\in K~|~(j,k)\in \bar{U}_{1,r+1},i=1,\ldots,m\right\};\\
&\vdots\\
V_{s_1,t_1}&=\left\{a_{jk}^{(i)}\in K~|~(j,k)\in\bar{U}_{s_1,t_1},i=1,\ldots,m\right\};\\
V_{s,r+s+t}&=\left\{a_{jk}^{(i)}\in K~|~(j,k)\in \bar{U}_{s,r+s+t},i=1,\ldots,m\right\}
\end{split}
\end{eqnarray*}
such that
\[
f(p_{1,r+1}(a_{jk}^{(i)}),\ldots,p_{s_1,t_1}(a_{jk}^{(i)}),p_{s,r+s+t}(a_{jk}^{(i)}),a_{s_2,t_2}',\ldots,a_{1,n}')\neq 0.
\]

We prove the claim by induction on $(s,r+s+t)$. We first consider the case of $(1,r+1)$. We take $a_{jk}^{(i)}\in K$ for all $(j,k)\in \bar{U}_{1,r+1}$, $i=1,\ldots,m$, where
\[
\left\{
\begin{aligned}
a_{s,s+1}^{(i_s')}&=1,\quad\mbox{for all $s=1,\ldots,r-1$};\\
a_{r,r+1}^{(i_r')}&=p_{i_1,\ldots,i_r}(\bar{b}_{1},\ldots,\bar{b}_{r+1})^{-1}a_{1,r+1}';\\
a_{jk}^{(i)}&=0,\quad\mbox{otherwise}.
\end{aligned}
\right.
\]

We get from (\ref{t1}) that
\begin{eqnarray*}
\begin{split}
p_{1,r+1}(a_{jk}^{(i)})&=p_{i_1',\ldots,i_r'}(\bar{b}_{1},\ldots,\bar{b}_{r+1})a_{12}^{(i_1')}\cdots a_{r,r+1}^{(i_r')}\\
&=a_{1,r+1}'.
\end{split}
\end{eqnarray*}

It follows from (\ref{e8}) that

\[
f(p_{1,r+1}(a_{jk}^{(i)}),a_{2,r+2}',\ldots,a_{s_1,t_1}',a_{s,r+s+t}',a_{s_2,t_2}',\ldots,a_{1,n}')\neq 0.
\]
This proves the case of $(1,r+1)$. Suppose that $(s,r+s+t)\neq (1,r+1)$. We consider the case of $(s,r+s+t)$. By induction on $(s_1,t_1)$,
we have that there exist the following compatible subsets of $K$
\begin{eqnarray*}
\begin{split}
V_{1,r+1}&=\left\{a_{jk}^{(i)}\in K~|~(j,k)\in \bar{U}_{1,r+1},i=1,\ldots,m\right\};\\
\vdots\\
V_{s_1,t_1}&=\left\{a_{jk}^{(i)}\in K~|~(j,k)\in \bar{U}_{s_1,t_1},i=1,\ldots,m\right\}
\end{split}
\end{eqnarray*}
such that
\begin{equation}\label{e10}
f(p_{1,r+1}(a_{jk}^{(i)}),\ldots,p_{s_1,t_1}(a_{jk}^{(i)}),a_{s,r+s+t}',a_{s_2,t_2}',\ldots,a_{1,n}')\neq 0.
\end{equation}

We now claim that
\[
(r+s-1,r+s+t)\in \bar{U}_{s,r+s+t}\setminus\left(\bar{U}_{1,r+1}\cup\cdots\cup \bar{U}_{s_1,t_1}\right).
\]
Indeed, we note that
\[
((r+s+t)-s)-((r+s+t)-(r+s-1))=r-1.
\]
In view of the definition of $\bar{U}_{s,r+s+t}$ we get that
\[
(r+s-1,r+s+t)\in \bar{U}_{s,r+s+t}.
\]
It suffices to prove that
\[
(r+s-1,r+s+t)\not\in \bar{U}_{1,r+1}\cup\cdots\cup \bar{U}_{s_1,t_1}.
\]

Suppose first that $s=1$. Note that $(s_1,t_1)=(n-r-t+1,n)$. Since
\[
(t_1-s_1)-((r+s+t)-(r+s-1))=r-2,
\]
we get that $(r+s-1,r+s+t)\not\in \bar{U}_{s_1,t_1}$.

For any $\bar{U}_{i,j}\in\{\bar{U}_{1,r+1},\ldots,\bar{U}_{s_1,t_1}\}$,
it is clear that $(t_1-s_1)\geq (j-i)$. We have that
\[
(j-i)-((r+s+t)-(r+s-1))\leq r-2.
\]
This implies that $(r+s-1,r+s+t)\not\in \bar{U}_{i,j}$. We have that
\[
(r+s-1,r+s+t)\not\in \bar{U}_{1,r+1}\cup\cdots\cup \bar{U}_{s_1,t_1}.
\]

Suppose next that $s>1$. Note that $(s_1,t_1)=(s-1,r+s-1+t)$. Since
\[
r+s+t>r+i+t
\]
for all $i=1,\ldots,s-1$, we see that
\[
(r+s-1,r+s+t)\not\in \bar{U}_{1,r+1+t}\cup \cdots \cup \bar{U}_{s_1,t_1}.
\]

Note that
\[
(n-(n-r-t+1))-((r+s+t)-(r+s-1))=r-2.
\]
We get that
\[
(r+s-1,r+s+t)\not\in \bar{U}_{n-r-t+1,n}.
\]

For any $\bar{U}_{i,j}\in\{\bar{U}_{1,r+1},\ldots,\bar{U}_{n-r-t+1,n}\}$, it is clear that
\[
n-(n-r-t+1)\geq j-i.
\]
We get that
\[
(j-i)-((r+s+t)-(r+s-1))\leq r-2.
\]
This implies that
\[
(r+s-1,r+s+t)\not\in \bar{U}_{i,j}.
\]
We have that
\[
(r+s-1,r+s+t)\not\in \bar{U}_{1,r+1}\cup\cdots\cup \bar{U}_{n-r-t+1,n}.
\]
Note that
\[
\{\bar{U}_{1,r+1},\ldots, \bar{U}_{s_1,t_1}\}=\{\bar{U}_{1,r+1},\ldots,\bar{U}_{n-r-t+1,n},\bar{U}_{1,r+1+t},\ldots,\bar{U}_{s_1,t_1}\}.
\]
We obtain that
\[
(r+s-1,r+s+t)\not\in \bar{U}_{1,r+1}\cup\cdots\cup \bar{U}_{s_1,t_1}.
\]
This proves the claim. We have that
\[
\bar{U}_{s,r+s+t}\setminus\left(\bar{U}_{1,r+1}\cup\cdots\cup \bar{U}_{s_1,t_1}\right)\neq\emptyset.
\]

We next claim that
\[
(s,s+1)\in \bar{U}_{s,r+s+t}\cap\left(\bar{U}_{1,r+1}\cup\cdots\cup \bar{U}_{s_1,t_1}\right).
\]

Note that
\[
((r+s+t)-s)-((s+1)-s)=r+t-1\geq r-1.
\]
In view of the definition of $\bar{U}_{s,r+s+t}$ we get that $(s,s+1)\in \bar{U}_{s,r+s+t}$. It suffices to prove that
\[
(s,s+1)\in \bar{U}_{1,r+1}\cup\cdots\cup \bar{U}_{s_1,t_1}.
\]

Suppose first that $s=1$. We note that
\[
(s_1,t_1)=(n-r-t+1,n)
\]
and
\[
\bar{U}_{1,r+t}\in\left\{\bar{U}_{1,r+1},\ldots,\bar{U}_{s_1,t_1}\right\}.
\]
It is clear that $(s,s+1)\in \bar{U}_{1,r+t}$. We get that
\[
(s,s+1)\in \bar{U}_{1,r+1}\cup\cdots\cup \bar{U}_{s_1,t_1}.
\]

Suppose next that $s>1$. Note that
\[
(s_1,t_1)=(s-1,r+s-1+t).
\]
It is clear that $(s,s+1)\in\bar{U}_{s_1,t_1}$. This implies that
\[
(s,s+1)\in \bar{U}_{1,r+1}\cup\cdots\cup \bar{U}_{s_1,t_1}.
\]
This proves the claim. We have that
\[
\bar{U}_{s,r+s+t}\cap\left(\bar{U}_{1,r+1}\cup\cdots\cup \bar{U}_{s_1,t_1}\right)\neq\emptyset.
\]

Let $p_{s,r+s+t}'$ be the sum of all monomials of $p_{s,r+s+t}$ that contains at least an variable $x_{jk}^{(i)}$ with $(j,k)\in \bar{U}_{s,r+s+t}\setminus\left(\bar{U}_{1,r+1}\cup\cdots\cup \bar{U}_{s_1,t_1}\right)$, for some $1\leq i\leq m$. Let
$p_{s,r+s+t}''$ be the sum of all monomials of $p_{s,r+s+t}$ on variables $x_{jk}^{(i)}$
with $(j,k)\in \bar{U}_{s,r+s+t}\cap\left(\bar{U}_{1,r+1}\cup\cdots\cup \bar{U}_{s_1,t_1}\right)$, $i=1,\ldots,m$. It is clear that
\[
p_{s,r+s+t}=p_{s,r+s+t}'+p_{s,r+s+t}''.
\]

Set
\[
\lambda_0=p_{i_1',\ldots,i_r'}(\bar{b}_{s},\ldots,\bar{b}_{r+s-1},\bar{b}_{r+s+t}).
\]
It follows from (\ref{e7}) that $\lambda_0\neq 0$. We take $a_{jk}^{(i)}\in K$ for all $(j,k)\in \bar{U}_{s,r+s+t}$, $i=1,\ldots,m$, where
\[
\left\{
\begin{aligned}
a_{w,w+1}^{(i_w')}&=1,\quad\mbox{for all $w=s,\ldots,r+s-2$};\\
a_{r+s-1,r+s+t}^{(i_{r}')}&=x_{r+s-1,r+s+t}^{(i_{r}')};\\
a_{jk}^{(i)}&=0,\quad\mbox{otherwise}.
\end{aligned}
\right.
\]
We get from (\ref{t1}) that

\begin{eqnarray}\label{e11}
\begin{split}
p_{s,r+s+t}'(a_{jk}^{(i)})’+p_{s,r+s+t}''(a_{jk}^{(i)})&=p_{s,r+s+t}(a_{jk}^{(i)})\\
&=\lambda_0x_{r+s-1,r+s+t}^{(i_r')}.
\end{split}
\end{eqnarray}

Since $(r+s-1,r+s+t)\in \bar{U}_{s,r+s+t}\setminus\left(\bar{U}_{1,r+1}\cup\cdots\cup \bar{U}_{s_1,t_1}\right)$, we get from (\ref{e11}) that
\begin{equation}\label{ee1}
p_{s,r+s+t}'(a_{jk}^{(i)})=\lambda_0x_{r+s-1,r+s+t}^{(i_r')}.
\end{equation}

Take $x_{r+s-1,r+s+t}^{(i_r')}=1$ in (\ref{ee1}). We have that

\begin{equation}\label{e12}
p_{s,r+s+t}'(a_{jk}^{(i)})=\lambda_0\neq 0.
\end{equation}

On the one hand, we fix $a_{jk}^{(i)}\in K$ for all $(j,k)\in \bar{U}_{s,r+s+r}\setminus\left(\bar{U}_{1,r+1}\cup\cdots\cup \bar{U}_{s_1,t_1}\right)$, $i=1,\ldots,m$, in (\ref{e12}).
We set
\[
g_{1}(x_{jk}^{(i)})=p_{s,r+s+t}'(x_{jk}^{(i)})
\]
for all $(j,k)\in \bar{U}_{s,r+s+t}\cap \left(\bar{U}_{1,r+1}\cup\cdots\cup \bar{U}_{s_1,t_1}\right)$, $i=1,\ldots,m$. It follows from (\ref{e12}) that
\[
g_{1}(a_{jk}^{(i)})=\lambda_0\neq 0.
\]
This implies that $g_1\neq 0$.

On the other hand, we fix $a_{jk}^{(i)}\in K$ for all $(j,k)\in \bar{U}_{s,r+s+r}\setminus\left(\bar{U}_{1,r+1}\cup\cdots\cup \bar{U}_{s_1,t_1}\right)$,
$i=1,\ldots,m$, in (\ref{e10}). We set
\[
g_{2}(x_{jk}^{(i)})=f(p_{1,r+1}(x_{jk}^{(i)}),\ldots,p_{s_1,t_1}(x_{jk}^{(i)}),a_{s,r+s+t}',a_{s_2,t_2}',\ldots,a_{1,n}')
\]
for all $(j,k)\in \bar{U}_{s,r+s+t}\cap\left(\bar{U}_{1,r+1}\cup\cdots\cup \bar{U}_{s_1,t_1}\right)$, $i=1,\ldots,m$. It follows from (\ref{e10}) that
\[
g_{2}(a_{jk}^{(i)})\neq 0.
\]
This implies that $g_2\neq 0$. In view of Lemma \ref{L6} we have that there exist $a_{jk}^{(i)}\in K$, for all
$(j,k)\in \bar{U}_{s,r+s+t}\cap\left(\bar{U}_{1,r+1}\cup\cdots\cup \bar{U}_{s_1,t_1}\right)$, $i=1,\ldots,m$, such that
\[
g_{1}(a_{jk}^{(i)})\neq 0\quad\mbox{and}\quad g_{2}(a_{jk}^{(i)})\neq 0.
\]
That is
\begin{equation}\label{e17}
p_{s,r+s+t}'(a_{jk}^{(i)})\neq 0
\end{equation}
and
\begin{equation}\label{e18}
f(p_{1,r+1}(a_{jk}^{(i)}),\ldots,p_{s_1,t_1}(a_{jk}^{(i)}),a_{s,r+s+t}',a_{s_2,t_2},\ldots,a_{1,n}')\neq 0.
\end{equation}

We set
\begin{equation}\label{w}
\alpha=p_{s,r+s+t}''(a_{jk}^{(i)}).
\end{equation}

Fix $a_{jk}^{(i)}\in K$ for all $(j,k)\in \bar{U}_{s,r+s+t}\cap\left(\bar{U}_{1,r+1}\cup\cdots \cup \bar{U}_{s_1,t_1}\right)$, $i=1,\ldots,m$, in (\ref{e17}). We set
\[
g_{3}(x_{jk}^{(i)})=p_{s,r+s+t}'(x_{jk}^{(i)})
\]
for all $(j,k)\in \bar{U}_{s,r+s+t}\setminus\left(\bar{U}_{1,r+1}\cup\cdots\cup \bar{U}_{s_1,t_1}\right)$, $i=1,\ldots,m$. It follows from (\ref{e17}) that
\[
g_{3}(a_{jk}^{(i)})\neq 0.
\]
We have that $g_3$ is a nonzero polynomial with zero constant term. In view of Lemma \ref{L3} we get that $g_{3}(K)=K$. Thus,
we have that there exist $a_{jk}^{(i)}\in K$, $(j,k)\in \bar{U}_{s,r+s+t}\setminus\left(\bar{U}_{1,r+1}\cup\cdots\cup \bar{U}_{s_1,t_1}\right)$, $i=1,\ldots,m$, such that
\[
g_{3}(a_{jk}^{(i)})=a_{s,r+s+t}'-\alpha.
\]
That is
\begin{equation}\label{w3}
p_{s,r+s+t}'(a_{jk}^{(i)})=a_{s,r+s+t}'-\alpha.
\end{equation}

It follows from both (\ref{w}) and (\ref{w3}) that
\begin{eqnarray*}
\begin{split}
p_{s,r+s+t}(a_{jk}^{(i)})&=p_{s,r+s+t}'(a_{jk}^{(i)})+p_{s,r+s+t}''(a_{jk}^{(i)})\\
&=(a_{s,r+s+t}'-\alpha)+\alpha\\
&=a_{s,r+s+t}'.
\end{split}
\end{eqnarray*}

We get from (\ref{e18}) that

\begin{equation}\label{e20}
f(p_{1,r+1}(a_{jk}^{(i)}),\ldots,p_{s_1,t_1}(a_{jk}^{(i)}),p_{s,r+s+t}(a_{jk}^{(i)}),a_{s_2,t_2},\ldots,a_{1,n}')\neq 0.
\end{equation}

We set
\[
V_{s,r+s+t}=\left\{a_{jk}^{(i)}\in K~|~(j,k)\in \bar{U}_{s,r+s+t},i=1,\ldots,m\right\}.
\]
Note that $V_{1,r+1},\ldots V_{s_1,t_1},V_{s,r+s+r}$ are compatible subsets of $K$.  This proves our claim.

Take $(s,r+s+t)=(1,n)$ in (\ref{e20}). We obtain that there exist the following compatible subsets of $K$
\begin{eqnarray*}
\begin{split}
V_{1,r+1}&=\left\{a_{jk}^{(i)}\in K~|~(j,k)\in \bar{U}_{1,r+1},i=1,\ldots,m\right\};\\
&\vdots\\
V_{1,r+2}&=\left\{a_{jk}^{(i)}\in K~|~(j,k)\in \bar{U}_{n-r,n},i=1,\ldots,m\right\};\\
&\vdots\\
V_{1,n}&=\left\{a_{jk}^{(i)}\in K~|~(j,k)\in \bar{U}_{1,n},i=1,\ldots,m\right\}
\end{split}
\end{eqnarray*}
such that
\begin{equation}\label{e21}
f(p_{1,r+1}(a_{jk}^{(i)}),\ldots,p_{1,r+2}(a_{jk}^{(i)}),\ldots,p_{1,n}(a_{jk}^{(i)}))\neq 0.
\end{equation}

We set
\[
c_{jj}^{(i)}=b_j^{(i)}
\]
for all $i=1,\ldots, m$ and $j=1,\ldots,n$ and
\[
c_{jk}^{(i)}=a_{jk}^{(i)}
\]
for all $1\leq j<k\leq n$, $i=1,\ldots,m$. We take
\[
u_i=(c_{jk}^{(i)})\in T_n(K)
\]
for all $i=1,\ldots,m$. It follows from (\ref{e2}) that
\[
p(u_1,\ldots,u_m)=(p_{1,r+1}(a_{jk}^{(i)}),\ldots,p_{1,r+2}(a_{jk}^{(i)}),\ldots,p_{1,n}(a_{jk}^{(i)})).
\]
We get from (\ref{e21}) that
\[
f(p(u_1,\ldots,u_m))\neq 0.
\]
This implies that $p(u_1,\ldots,u_m)\in U$. We have that $p(T_n(K))\cap U\neq \emptyset$. This implies that
$p(T_n(K))$ is a dense subset of $T_n(K)^{(r-1)}$ (with respect to the Zariski topology). This proves the result.
\end{proof}


\begin{thebibliography}{stringX}
\bibitem{Albert} A. Albert, B. Mukenhoupt, On matrices of trace zero, Michigan Math. J. 4 (1957) 1--3.

\bibitem{An} B.\,E. Anzis, Z.\,M. Emrich, K.\,G. Valiveti, On the images of Lie
polynomials evaluated on Lie algebras, Linear Algebra Appl. 469 (2015) 51--75.

\bibitem{Rowen1} A. Kanel-Belov, S. Malev, L. Rowen, The images of non-commutative polynomials evaluated on
$2\times 2$ matrices, Proc. Amer. Math. Soc. 140 (2012) 465--478.

\bibitem{Rowen2} A. Kanel-Belov, S. Malev, L. Rowen, The images of multilinear
polynomials evaluated on $3\times 3$ matrices, Proc. Amer. Math.
Soc. 144 (2016) 7--19.

\bibitem{Rowen3} A. Kanel-Belov, S. Malev, L. Rowen, Power-central polynomials
on matrices, J. Pure Appl. Algebra 220 (2016) 2164--2176.

\bibitem{Rowen4} A. Kanel-Belov, S. Malev, L. Rowen, The images of Lie polynomials
evaluated on matrices, Comm. Algebra 45 (2017) 4801--4808.

\bibitem{survey} A. Kanel-Belov, S. Malev, L. Rowen, R. Yavich, Evaluations of noncommutative polynomials on algebras: methods and problems,
and the Lvov-Kaplansky conjecture, SIGMA. 16 (2020) 071.

\bibitem{Rowen2022}  A. Kanel-Belov, S. Malev, C. Pines, L. Rowen, The images of multilinear and
semihomogeneous polynomials on the algebra of octonions, arXiv:2204.07139v1.[math.AG].

\bibitem{Bre1} M. Bre\v{s}ar, Introduction to noncommutative algebra, Springer, New
York, 2014.

\bibitem{Bre} M. Bre\v{s}ar, Commutators and images of noncommutative polynomials, Adv. Math. 374 (2020) 107--140.

\bibitem{B1} M. Bre\v{s}ar, P. \v{S}emrl, The Waring problem for matrix algebras, Israel J. Math. 139 (2022) 1--25.

\bibitem{B2} M. Bre\v{s}ar, P. \v{S}emrl, The Waring problem for matrix algebras, II, arXiv:2302.05106, accepted for publication in Bull. London Math. Soc.

\bibitem{BW} D. Buzinski, R. Winstanley, On multilinear polynomials in four variables evaluated on matrices, Linear Algebra Appl. 439 (2013) 2712--2719.

\bibitem{ChenLuoWang} Q. Chen, Y.\,Y. Luo, Y. Wang, The image of polynomials on $3\times 3$ upper triangular matrix algebras,
Linear Algebra Appl. 648 (2022) 254--269.

\bibitem{Chuang} C.\,-L. Chuang, On ranges of polynomials in finite matrix rings, Proc. Amer. Math. Soc. 110 (1990) 293--302.

\bibitem{Dyk} K.\,J. Dykema, I. Klep, Instances of the Kaplansky-Lvov multilinear conjecture for polynomials of degree
three, Linear Algebra Appl. 508 (2016) 272--288.

\bibitem{Formanek} E. Formanek, Central polynomials for matrix rings, J. Algebra 23 (1972) 129--132.

\bibitem{Do} Dniester Notebook: Unsolved problems in the theory of rings and modules, 4th ed., Mathematics Institute,
Russian Academy of Sciences Siberian Branch, Novosibirsk, 1993.

\bibitem{Drensky} V. Drensky, Free algebras and PI-algebras, Graduate Course in Algebras, Hong Kong, 1996, 1--197.

\bibitem{DCM} V. Drensky, P. Cattaneo, G. Maria, A central
polynomial of low degree for $4\times 4$ matrices, J. Algebra 168 (1994) 469--478.

\bibitem{Fag} P.\,S. Fagundes, The images of multilinear polynomials
on strictly upper triangular matrices, Linear Algebra Appl. 563 (2019) 287--301.

\bibitem{FM} P.\,S. Fagundes,  TC. de Mello, Images of multilinear
polynomials of degree up to four on upper triangular matrices, Oper
Matrices 13 (2019) 283--292.

\bibitem{FK} P.\,S. Fagundes, P. Koshlukov, Images of multilinear graded polynomials on upper triangular matrix algebras, Canadian J. Math. Published online by Cambridge University Press: 19 September 2022, pp. 1--26.

\bibitem{Mello} I.\,G. Gargate, T.\,C. de Mello, Images of multilinear
polynomials on $n\times n$ upper triangular matrices over infinite field, Israel J. Math. 252 (2022) 337--354.

\bibitem{Robin} R. Hartshorne, \emph{Algebraic Geometry}, Graduate Texts in Mathematics, 52, Springer-Verlag, New York-Berlin-Heodelberg, 1997.

\bibitem{Hui} C.\,Y. Hui, M. Larsen, A. Shalev, The Waring problem for Lie groups and Chevalley groups, Israel J. Math. 210 (2015) 81--100.

\bibitem{Jac} N. Jacobson, Basic Algebra I, Second Edition, W. H. Freeman and Company, New York, 1985.

\bibitem{Kul1} V.\,V. Kulyamin, Images of graded polynomials in matrix rings over finite group algebras, Russian Math.
Surveys 55 (2000) 345--346.

\bibitem{Kul2} V.\,V. Kulyamin, \emph{On images of polynomials in finite matrix rings}, Ph.D. Thesis, Moscow Lomonosov State
University, Moscow, 2000.

\bibitem{LuoChen} Y.\,Y. Luo, Q. Chen, The images of linear polynomials with zero constant term on upper triangular matrix algebras, arXiv:2304.01591v1, 4 Apr 2023.
\bibitem{Wang2022} Y.\,Y. Luo, Y. Wang, On Fagundes-Mello conjecture, J. Algebra  592 (2022) 118--152.

\bibitem{Ma} A. Ma, J. Oliva, On the images of Jordan polynomials evaluated
over symmetric matrices, Linear Algebra Appl. 492 (2016) 13--25.

\bibitem{Malev1} S. Malev, The images of non-commutative polynomials evaluated
on $2\times 2$ matrices over an arbitrary field, J. Algebra Appl. 13 (2014) 145--154.

\bibitem{Malev2} S. Malev, The images of noncommutative polynomials evaluated on the quaternion algebra, J. Algebra and Appl. 20 (2021) 2150074.

\bibitem{Mello1} TC. de Mello, The image of multilinear polynomials evaluated on $3\times 3$ upper triangular matrices, Comm. Math. 29 (2021) 183--186.

\bibitem{Mes} Z. Mesyan, Polynomials of small degree evaluated on matrices, Linear Multilinear Algebra 61 (2013) 1487--1495.

\bibitem{PP}  S. Panja and S. Prasad, The image of polynomials and Waring type problems on upper triangular matrix algebras, J. Algebra 631 (2023) 148--193.

\bibitem{Sho} K. Shoda, Einige S$\ddot{a}$tze $\ddot{u}$ber Matrizen, Jap. J. Math. 13 (1936) 361--365.

\bibitem{SP} $\breve{S}$. $\breve{S}$penko, On the image of a noncommutative polynomial, J. Algebra 377 (2013) 298--311.

\bibitem{Vitas1} D. Vitas, Images of multilinear polynomials in the algebra of finitary matrices contain trace zero matrices,
Linear Algebra Appl. 626 (2021) 221--233.

\bibitem{Vitas2} D. Vitas, Multilinear polynomials are surjective on algebras with surjective inner derivations,
J. Algebra 565 (2021) 255--281.

\bibitem{Wang2019} Y. Wang, The images of multilinear polynomials on $2\times 2$
upper triangular matrix algebras, Linear Multilinear Algebra 67 (2019) 2366--2372.

\bibitem{Wang20191} Y. Wang, P.\,P. Liu, J. Bai, Correction: The images of multilinear polynomials on $2\times 2$
upper triangular matrix algebras, Linear Multilinear Algebra 67 (2019) i--vi.

\bibitem{WangZL} Y. Wang, J. Zhou, Y.\,Y. Luo, The image of polynomials on $2\times 2$ upper triangular matrix algebras,
Linear Algebra Appl. 610 (2021) 560--573.

\bibitem{ZhouWang} J. Zhou, Y. Wang, The image of complete homogenous polynomials on $2\times 2$ upper triangular matrix algebras,
Algebra Represent. Theory 24 (2021) 1221--1229.

\end{thebibliography}
\end{document}